\documentclass[11pt]{article} 
\usepackage[utf8]{inputenc}
\usepackage{amsfonts,amsmath,amsthm,amssymb}
\usepackage[justification=centering]{caption}
\usepackage{color,soul}
\usepackage{fullpage}
\usepackage{enumitem}   
\usepackage{mathtools}
\usepackage{xcolor}
\usepackage{thmtools}
\usepackage{float}
\usepackage{amsthm}

\usepackage{amsmath,stackengine}
\usepackage{hyperref}
\usepackage{cleveref}

\usepackage{thm-restate}
\makeatletter
\newcommand{\addresseshere}{%
  \enddoc@text\let\enddoc@text\relax
}
\makeatother

\newtheorem{theorem}{Theorem}[section]
\newtheorem{proposition}[theorem]{Proposition}
\newtheorem{lemma}[theorem]{Lemma}
\newtheorem{corollary}[theorem]{Corollary}

\theoremstyle{definition}

\newtheorem{definition}[theorem]{Definition}

\usepackage{etoolbox}
\AtEndEnvironment{proof}{\setcounter{claim}{0}}
\newtheoremstyle{notation}
{1em}
{1em}
{}
{}
{\bfseries}
{.}
{.5em} 
{} 
\theoremstyle{notation}

\newcommand*{\R}{\mathbb{R}}

\newcommand{\calA}{\mathcal{A}}

\newcommand{\hroot}{\tilde{\alpha}}
\newcommand*{\C}{\mathbb{C}}

\newtheorem{innercustomthm}{Theorem}
\newenvironment{customthm}[1] {\renewcommand\theinnercustomthm{#1}\innercustomthm}
  {\endinnercustomthm}

\newtheorem{innercustomcor}{Corollary}
\newenvironment{customcor}[1]{\renewcommand\theinnercustomcor{#1}\innercustomcor}
  {\endinnercustomcor}

\usepackage[colorinlistoftodos]{todonotes}

\title{ 
Computing the $q$-Multiplicity of the Positive Roots of $\mathfrak{sl}_{r+1}(\mathbb{C})$ and Products of Fibonacci Numbers
}

\usepackage{authblk}

\author[1]{Kimberly J. Harry}

\affil[1]{Department of Mathematical Sciences, University of Wisconsin - Milwaukee, \newline
Milwaukee, WI 53211\\
 \textcolor{blue}{\href{mailto:kjharry@uwm.edu}{kjharry@uwm.edu}}
}
\date{}

\begin{document}
\maketitle

\begin{abstract}
Using Kostant's weight multiplicity formula, we describe and enumerate the terms contributing a nonzero value to the multiplicity of a positive root $\mu$ in the adjoint representation of $\mathfrak{sl}_{r+1}(\mathbb{C})$, which we denote $L(\tilde{\alpha})$, where $\tilde{\alpha}$ is the highest root of $\mathfrak{sl}_{r+1}(\mathbb{C})$.
We prove that the number of terms contributing a nonzero value to the multiplicity of the positive root $\mu=\alpha_i+\alpha_{i+1}+\cdots+\alpha_j$ with $1\leq i\leq j\leq r$ in $L(\tilde{\alpha})$ is given by the product $F_{i}\cdot F_{r-j+1}$, where $F_n$ is the $n^{\text{th}}$ Fibonacci number.
Using this result, we show that the $q$-multiplicity of the positive root $\mu=\alpha_i+\alpha_{i+1}+\cdots+\alpha_j$ with $1\leq i\leq j\leq r$ in the representation $L(\tilde{\alpha})$ is precisely $q^{r-h(\mu)}$, where $h(\mu)=j-i+1$ is the height of the positive root $\mu$. 
Setting $q=1$ recovers the known result that the multiplicity of a positive root in the adjoint representation of $\mathfrak{sl}_{r+1}(\mathbb{C})$ is one.
\end{abstract}

\section{Introduction} 
For $r\geq 1$, the type $A_r$ Lie algebra, also known as the special linear algebra, is denoted  $\mathfrak{sl}_{r+1}(\C) = \{X \in M_{r+1}(\C): \text{Tr}(X) = 0\}$. We recall that   $\mathfrak{sl}_{r+1}(\C)$
is a complex vector space with a bilinear product called the Lie bracket, defined by the commutator bracket $[X,Y]=XY-YX$ for $X,Y\in\mathfrak{sl}_{r+1}(\C)$, which is skew-symmetric and satisfies the Jacobi identity.
Given $e_1,e_2,\ldots,e_{r+1}$, the standard basis vectors of $\R^{r+1}$, for each $i\in[r]$, we define
$\alpha_i = e_i - e_{i+1}$ .
We let $\Delta=\{\alpha_1,\alpha_2,\ldots,\alpha_r\}$ be the set of simple roots, $\Phi^+=\Delta \cup \{\alpha_i+\alpha_{i+1}+\cdots+\alpha_j: 1\leq i< j\leq r\}$ be the set of positive roots, and $\tilde{\alpha}= \alpha_1+\alpha_2+\cdots +\alpha_r$ be the highest root of $\mathfrak{sl}_{r+1}(\C)$.

We recall that the Weyl group $W$ is the group generated by reflections, denoted as $s_1, s_2,\ldots, s_r$, through hyperplanes which are orthogonal to the simple roots $\alpha_1,\alpha_2,\ldots,\alpha_r$, respectively. 
The Weyl group of type $A_r$ is isomorphic to the symmetric group $\mathfrak{S}_{r+1}$.
If $\sigma\in W$, then $\ell(\sigma)=k$ is the length of $\sigma$, which means that $k$ is the smallest nonnegative integer such that $\sigma$ can be expressed (not necessarily uniquely) as a product of $k$ reflections.

One area of interest in combinatorial representation theory is the
computation involved in determining the multiplicity of a weight $\mu$ in a highest weight representation of $\mathfrak{sl}_{r+1}(\mathbb{C})$ 
with highest weight $\lambda$.
One way to compute this multiplicity is via Kostant's weight multiplcity formula. 
For weights $\lambda$ and $\mu$, let the \textbf{Kostant's weight multiplicity formula} to be 
\begin{align}
    m(\lambda, \mu)=\sum_{\sigma \in W} (-1)^{\ell(\sigma)}\wp(\sigma(\lambda+\rho)-\mu - \rho)\label{KWMF}
\end{align}
where $\rho = \frac{1}{2}\sum_{\alpha \in \Phi^{+}} \alpha$ and 
$\wp(\xi)$ is Kostant's partition function,
which counts the number of ways the weight $\xi$ can be written as a nonnegative integral sum of positive roots (Theorem on page 589 
in \cite{Kostant}). 

When calculating the Kostant's weight multiplicity formula \eqref{KWMF}, every element in the Weyl group contributes a term to the multiplicity.
We know the order of the Weyl group increases factorially as the rank of the Lie algebra increases, and it has been observed that many of the terms contribute trivially (i.e. give a zero term) to the computation \cite{Cochet}.
To formulate a more efficient way to compute Kostant's weight multiplicity formula, Harris gave the following definition.
\begin{definition}[Definition 1.1 in \cite{PHThesisPublication}]
    For $\lambda$ and $\mu$ integral weights of $\mathfrak{g}$, let the \textbf{Weyl alternation set} to be
    \begin{equation}
        \mathcal{A}(\lambda, \mu) = \{\sigma \in W ~|~ 
 \wp(\sigma(\lambda+\rho)-\mu - \rho)>0\}.
    \end{equation}
\end{definition} 
This set tells us the elements of the Weyl group for which $\sigma(\lambda+\rho)-\mu - \rho$ can be written as a nonnegative integral sum of positive roots.
Knowing the Weyl alternation set allows one to reduce the weight multiplicity computation by only needing to consider those Weyl group elements that contribute nonzero terms to this computation.

An additional challenge in computing weight multiplicities is that there are no general formulas for Kostant's partition function nor for its $q$-analog.
We recall the following result of Lusztig, which defines the $q$-analog of Kostant's weight multiplicity formula.

\begin{definition}[Proposition 9.2 in \cite{Lusztig}]
    The \textbf{$q$-analog of Kostant's partition function} is the polynomial-valued function 
    \begin{align}\label{KPF}
    \wp_q(\xi)=c_0+c_1q+c_2q^2+\cdots + c_kq^k,
    \end{align} 
    where $c_i $ equals the number of ways to write $\xi$ as a sum of exactly $i$ positive roots.
    Then the \textbf{$q$-analog of Kostant's weight multiplicity formula} is given by
\begin{align}
    m_q(\lambda, \mu)=\sum_{\sigma \in W} (-1)^{\ell(\sigma)}\wp_q(\sigma(\lambda+\rho)-\mu - \rho).\label{q-analog}
\end{align}
\end{definition}
Observe that evaluating equation \eqref{q-analog} at $q=1$ recovers the multiplicity in equation \eqref{KWMF}.

Work on characterizing and enumerating the Weyl alternation sets for certain pairs of weights, $(\lambda, \mu)$, has gathered some attention in the literature. 
This includes the following. 
If $\lambda$ and $\mu$ are integral weights of the Lie algebras $\mathfrak{sl}_{3}(\mathbb{C})$ and $\mathfrak{sl}_{4}(\mathbb{C})$, then the sets $\calA(\lambda,\mu)$ are known in general \cite{Lucy,Mabie}.
Harris considers $\hroot$, the highest root of $\mathfrak{sl}_{r+1}(\mathbb{C})$ with $r\geq 1$, and $\mu=0$ and establishes the cardinality of the Weyl alternation set, $|\mathcal{A}(\hroot,0)|$, is a Fibonacci number \cite[Theorem~1.2]{PHThesisPublication}.
Harris also states that if $\mu$ is a nonzero dominant weight,
then the Weyl alternation set is either $\{1\}$ when $\mu = \hroot$ or empty otherwise.
Chang, Harris, and Insko extend these results and show that  when $\mu=0$ and $\lambda$ is the sum of all the simple roots of the simple Lie algebra  $\mathfrak{so}_{2r+1}(\mathbb{C})$,
the cardinality of the Weyl alternation set $\mathcal{A}(\lambda,0)$ is given by a Fibonacci number \cite[Corollary 2.1 and 3.1]{Kevin}.
While in the Lie algebras
$\mathfrak{sp}_{2r}(\mathbb{C})$ and $\mathfrak{so}_{2r}(\mathbb{C})$, the cardinalities of the analogous Weyl alternation sets $\mathcal{A}(\lambda,0)$ are given by a multiple of a Lucas number \cite[Corollary 4.1 and 5.1]{Kevin}.
Harris, Insko, and Williams consider $\tilde{\alpha}$, the highest root of the classical Lie algebras, and $\mu=0$ and show that $|\calA(\hroot,0)|$ satisfies homogeneous linear recurrence relations \cite{Lauren}. 
In the type $B$ case, if $\mu=\varpi_1$ denotes the first fundamental root, then $|\calA(\lambda,\varpi_1)|$
is a Fibonacci number \cite[Corollary 7.4]{Lauren}. 
They also consider the case where $\mu$ is a dominant positive root and give the multiplicity of these weights in the adjoint representation of the classical Lie algebras.
Berenshte\u{\i}n and Zelevinski\u{\i} \cite[Theorem 1.3]{equalto1} provide a comprehensive list of pairs of weights $(\lambda,\mu)$ of the classical Lie algebras and the exceptional Lie algebra $G_2$ for which $m(\lambda,\mu)=1$. 
Harris, Rahmoeller, Schneider, and Simpson consider certain such pairs,
determine $\calA(\lambda,\mu)$, and show that the $q$-multiplicity is a power of $q$ \cite{powersofq}.
For all pairs of weights that Berenshte\u{\i}n and Zelevinski\u{\i} provide,
Harris et. al~conjecture that 
$m_q(\lambda,\mu)$ is a power of $q$ \cite[Conjecture 7]{powersofq}, and as of yet, this conjecture remains an open problem.

As noted above, past work fixed $\lambda$ to be the highest root or the sum of the simple roots of a classical Lie algebra, while $\mu$ was either $0$ or a dominant positive root. 
In this work, we focus our study on
$\mathfrak{sl}_{r+1}(\mathbb{C})$,
fix $r \geq 1$, let $\hroot$ be the highest root, and consider $\mu \in \Phi^+$.
In this way, $\mu$ may be a nondominant weight.
It is known that whenever $\mu$ is a root, $m(\hroot,\mu)=1$. 
Hence, when $\mu$ is dominant, the pair $(\lambda,\mu)$ is given by Berenshte\u{\i}n and Zelevinski\u{\i} \cite{equalto1}. 
Moreover, we recall that the multiplicity is invariant under the action of the Weyl group. 
Namely, $m(\hroot,\mu)=m(\hroot,\sigma(\mu))$ for all $\sigma\in W$ \cite[Proposition 3.1.20]{GW}. 
However, when $\mu\in\Phi^+$ is nondominant, a characterization and enumeration of the elements in $\mathcal{A}(\hroot,\mu)$ was unknown.
This paper contributes to this literature by proving the following results. 

\begin{customthm}{2.1}
    Fix $1\leq i \leq j\leq r$ and let $\hroot$ be the highest root of $\mathfrak{sl}_{r+1}(\C)$. 
    Then $\sigma \in \calA(\hroot, \alpha_i+\alpha_{i+1}+\cdots+\alpha_j)$
    if and only if 
    $\sigma=1$ or $\sigma=s_{\ell_1} \cdots s_{\ell_k} s_{\ell_{k+1}} \cdots s_{\ell_t}$ for some collection of nonconsecutive integers
    $2\leq \ell_1, \ell_2, \ldots, \ell_k \leq i-1$ and $j+1 \leq \ell_{k+1}, \ldots, \ell_{t} \leq r-1$.
\end{customthm} 
From the characterization in Theorem \ref{Nonconsecutive_Integers}, we are able to provide an enumeration for the cardinality of the Weyl alternation set.

\begin{customcor}{2.5}
    Fix $1\leq i\leq j\leq r$ and let $\hroot$ be the highest root of $\mathfrak{sl}_{r+1}(\C)$.
    Then 
    \[|\calA(\hroot, \alpha_i + \cdots + \alpha_j)| = F_i \cdot F_{r-j+1},\]
    where $F_n$ denotes the $n^{\text{th}}$ Fibonacci number
    \cite[\href{https://oeis.org/A000045}{A000045}]{OEIS}.
\end{customcor}

Given the characterization and enumeration of the terms contributing nontrivially to the $q$-analog of Kostant's weight multiplicity formula, we prove the following result using purely combinatorial techniques.

\begin{customthm}{4.1}
    Fix $1\leq i \leq j\leq r$ and let $\hroot$ be the highest root of $\mathfrak{sl}_{r+1}(\C)$. 
    Then \[m_q(\tilde{\alpha}, \mu)= q^{r-h(\mu)},\]
    where 
$h(\mu)=j-i+1$ denotes the height of $\mu$.  
\end{customthm}

Observe that setting $q = 1$ in Theorem \ref{Main Conjecture} recovers the known result that the multiplicity of a
positive root in the adjoint representation is equal to one \cite[Theorem~2.4.1]{GW}.

We remark that the work in \cite{Broer,Brylinski, Anthony}, together imply that when $\hroot$ is the highest root, $\mu$ is a dominant weight, and $w$ is a Weyl group element, the jump polynomial associated to the weight $w(\mu)$ is equal to a power of $q$ times the $q$-analog of Kostant's weight multiplicity formula $m_q(\lambda,\mu)$. 
In such a case, $m(\tilde{\alpha}, \mu) =1$, and $m_q(\tilde{\alpha},\mu)$ is a power of $q$. 
Thus, providing an alternative proof of Theorem \ref{Main Conjecture}.
However, the proofs presented in the current are purely combinatorial.

This work is organized as follows:
In Section \ref{Weyl alternation Section}, we present the proof to Theorem \ref{Weyl element thm} showing that the cardinality of the Weyl alternation set is a product of Fibonacci numbers.
In Section \ref{q-analog of KPF Section}, we establish results related to the $q$-analog of Kostant's partition function. 
In Section \ref{Main Conjecture for q-analog Section}, we present a proof to Theorem \ref{Main Conjecture} showing that the $q$-analog of Kostant's weight multiplicity formula
of a positive root in the adjoint  representation simplifies down to a power of $q$.
\bigskip

\noindent \textbf{Acknowledgements.} 
The author thanks her adviser Dr.~Pamela E.~Harris for the introduction to the problem and feedback on an earlier draft. 
She also thanks Dr.~Jeb Willenbring for insightful comments throughout the progression of this manuscript. The author gratefully acknowledges funding from the University of Wisconsin - Milwaukee Department of Mathematical Sciences in support of this research.

\section{Weyl Alternation Sets}\label{Weyl alternation Section}

In \cite{Kevin}, they consider the case where $\lambda$ is the highest root and $\mu=0$.
We still let $\lambda$ be the highest root and denote it as $\hroot$. 
However, we let $\mu \in \Phi^+$.
The following result characterizes the elements in the set $\mathcal{A}(\hroot, \mu)$.
The following theorem is closely related to Theorem 2.1 in \cite{Kevin}, and we adapt the proof techniques for our purpose.

\begin{theorem}\label{Nonconsecutive_Integers}
    Fix $1\leq i \leq j\leq r$ and let $\hroot$ be the highest root of $\mathfrak{sl}_{r+1}(\C)$. 
    Then $\sigma \in \calA(\hroot, \alpha_i+\alpha_{i+1}+\cdots+\alpha_j)$
    if and only if 
    $\sigma=1$ or $\sigma=s_{\ell_1} \cdots s_{\ell_k} s_{\ell_{k+1}} \cdots s_{\ell_t}$ for some collection of nonconsecutive integers
    $2\leq \ell_1, \ell_2, \ldots, \ell_k \leq i-1$ and $j+1 \leq \ell_{k+1}, \ldots, \ell_{t} \leq r-1$.\label{Weyl element thm}
\end{theorem}

\begin{proof}
    ($\Rightarrow$) We proceed by contradiction and assume $\sigma \neq 1$ or $\sigma \neq s_{\ell_1} \cdots s_{\ell_k} s_{\ell_{k+1}} \cdots s_{\ell_t}$ for some collection of nonconsecutive integers $2\leq \ell_1, \ell_2, \ldots, \ell_k \leq i-1$ and $j+1 \leq \ell_{k+1}, \ldots, \ell_{t} \leq r-1$.
    Then $\sigma$ must contain either $s_1$, $s_r$, $s_y$ with $i\leq y\leq j$, or some product of $s_n$ and $s_{n+1}$ for some integer $1\leq n\leq r$.
     
    Let $\sigma=s_1$.
    Then
        \begin{align*}
            \wp(s_1(\hroot + \rho) - \rho - (\alpha_i + \alpha_{i+1}+ \cdots + \alpha_j)) &= \wp(\hroot - \alpha_1 + \rho - \alpha_1 - \rho - (\alpha_i + \alpha_{i+1}+ \cdots + \alpha_j))\nonumber\\
            &= \wp(\hroot - 2\alpha_1 - (\alpha_i + \alpha_{i+1}+ \cdots + \alpha_j))\nonumber\\
            &= \wp(-\alpha_1 +\alpha_2 + \cdots + \alpha_{i-1} + \alpha_{j+1}+ \cdots + \alpha_r)\nonumber\\
            &=0
        \end{align*}
    since $-\alpha_1 +\alpha_2 + \cdots + \alpha_{i-1} + \alpha_{j+1}+ \cdots + \alpha_r$ cannot be written as a sum of positive roots since the coefficient for $\alpha_1$ is negative.
    Thus, we get a contradiction, and $s_1 \notin \calA(\hroot, \alpha_i+\alpha_{i+1}+\cdots+\alpha_j)$.
    By \cite[Proposition 3.4]{Lauren}, if $\sigma \notin \calA(\hroot, \alpha_i+\alpha_{i+1}+\cdots+\alpha_j)$, then neither is any $\sigma'\in W$ that contains $\sigma$ in its reduced word expression.
    Thus, $\sigma \in W$ containing $s_1$ in its reduced word expression cannot be in $\calA(\hroot, \alpha_i+\alpha_{i+1}+\cdots+\alpha_j)$.
    This reaches a contradiction to the assumption that $\sigma\in\calA(\hroot,\alpha_i+\cdots+\alpha_j)$.

    A similar argument is made for when $\sigma$ contains either $s_r$, $s_y$ with $i\leq y\leq j$, or 
    some product of $s_n$ and $s_{n+1}$ for some integer $1\leq n\leq r$.
 
We have shown that if $\sigma\in \calA(\hroot,\alpha_i+\cdots+\alpha_j)$, then $\sigma$ cannot contain $s_1$, $s_r$, $s_k$ with $i\leq k\leq j$, nor any product of consecutive $s_n$ and $s_{n+1}$ with $1\leq n\leq r-1$. 
Thus, the only remaining possibilities are that $\sigma=1$ or that $\sigma=s_{\ell_1} \cdots s_{\ell_k} s_{\ell_{k+1}} \cdots s_{\ell_t}$ for some collection of nonconsecutive integers
    $2\leq \ell_1, \ell_2, \ldots, \ell_k \leq i-1$ and $j+1 \leq \ell_{k+1}, \ldots, \ell_{t} \leq r-1$ as desired.

\noindent($\Leftarrow$) It suffices to consider the following two cases.

\noindent\textbf{Case 1}: If $\sigma=1$, then
        \begin{align*}
            \wp(1(\hroot+\rho)-\rho-(\alpha_i+\alpha_{i+1}+\cdots+\alpha_j)) &= \wp(\hroot - (\alpha_i+\alpha_{i+1}+\cdots+\alpha_j))\\
            &= \wp(\alpha_1 + \cdots +\alpha_{i-1}+\alpha_{j+1}+\cdots +\alpha_{r})\\
            &>0
        \end{align*}
    since $\alpha_1 + \cdots +\alpha_{i-1}+\alpha_{j+1}+\cdots +\alpha_{r}$ can be written as a sum of positive roots.
    Thus, \newline $\sigma=1 \in \calA(\hroot, \alpha_i+\alpha_{i+1}+\cdots + \alpha_j)$.
    
\noindent\textbf{Case 2}: Let $\sigma=s_{\ell_1} \cdots s_{\ell_k} s_{\ell_{k+1}} \cdots s_{\ell_t}$ where for some collection of nonconsecutive integers \newline
    $2\leq \ell_1, \ell_2, \ldots, \ell_k \leq i-1$ 
    and $j+1 \leq \ell_{k+1}, \ldots, \ell_{t} \leq r-1$.
    Then
    \begin{align}
            \wp(s_{\ell_1}\cdots s_{\ell_k} s_{\ell_{k+1}}\cdots s_{\ell_t}&(\hroot + \rho)-\rho - (\alpha_i + \alpha_{i+1}+\cdots + \alpha_{j}))\nonumber\\
            &=
            \wp\left(\hroot +\rho- 
        \sum_{h=1}^t\alpha_{\ell_h}
            -\rho-(\alpha_i + \alpha_{i+1}+\cdots + \alpha_{j})\right)\nonumber\\
             &=
             \wp\left(\hroot - (\alpha_i + \alpha_{i+1}+\cdots + \alpha_{j}) - \sum_{h=1}^t\alpha_{\ell_h}\right).\label{eq:subtract large sum}
    \end{align}
    Notice that 
    $\{i,i+1,\ldots, j\}\cap \{\ell_1,\ell_2,\ldots, \ell_t\}=\emptyset$.
    If $x\in\{i,i+1,\ldots,j\}\cup\{\ell_1,\ell_2,\ldots, \ell_t\}$, then in equation \eqref{eq:subtract large sum}, the coefficient of $\alpha_x$ is exactly 0.
    If $x\notin \{i,i+1,\ldots,j\}\cup\{\ell_1,\ell_2,\ldots, \ell_t\}$, then in equation \eqref{eq:subtract large sum}, the coefficient of $\alpha_x$ is exactly 1.
    Hence, $\hroot - (\alpha_i + \alpha_{i+1}+\cdots + \alpha_{j}) - \sum_{h=1}^t\alpha_{\ell_h}$, can be written as a sum of positive roots.
    Therefore, $s_{\ell_1} \cdots s_{\ell_k} s_{\ell_{k+1}} \cdots s_{\ell_t} \in \calA(\hroot, \alpha_i+\alpha_{i+1}+\cdots + \alpha_j)$.
\end{proof}

Having characterized the elements of $\calA(\hroot,\alpha_i+\cdots+\alpha_j)$ for all $1\leq i\leq j\leq r$, we now consider the cardinality of the set. To do so requires the following definition and enumerative result.

We begin by defining the sequence of numbers commonly referred to as Fibonacci \cite{sigler2003fibonacci} but was first introduced in the literature by Pingala \cite{Pingala}. 

\begin{definition}
    For any $n\geq 3$, define the \textbf{Fibonacci number}, denoted $F_n$, by the recursive formula: \[F_{n}=F_{n-1}+F_{n-2},\] with seed values $F_1=F_2=1$.
\end{definition}
 
We utilize the following enumerative result in later results.

\begin{proposition}
    The number of ways to select a subset of $[n]:=\{1,2,3,\ldots, n\}$ consisting of nonconsecutive numbers is given $F_{n+2}$, where $F_i$ is the $i^\text{th}$ Fibonacci number.
\label{Fib. Theorem}
\end{proposition}

\begin{proof}
We proceed by induction. 
For $n=1$, the only options are to select no integers or to select $1$.
There are two such choices and note $F_{1+2}=F_3=2$, and more over $F_3=F_2+F_1$, as claimed.
Assume that for all $k\leq n$ the number of ways to select a subset of $[k]$ consisting of only nonconsecutive integers is given by $F_{k+2}$.
Now let $k=n+1$.
Note that for any subset of $[n-1]$  consisting of nonconsecutive integers, we can include the integer $n+1$ to get a subset of $[n+1]$  consisting of nonconsecutive integers. 
Then by induction, there are $F_{n-1+2}=F_{n+1}$ many such subsets, all of which contain the integer $n+1$.
Moreover, any subset of $[n]$ consisting of nonconsecutive integers is also a subset of $[n+1]$ with the same property and does not contain $n+1$.
By induction hypothesis, there are $F_{n+2}$ such subsets. 
As these subsets are disjoint (since they either contain $n+1$ or not), then we have shown that the number of subsets of $[n+1]$ consisting of nonconsecutive integers is given by $F_{n+2}+F_{n+1}$, which is equal to $F_{n+3}$, as desired.
\end{proof}
In the following section, we use the enumerative result Harris gives in \cite{HarrisThesis} 
which counts subsets of $k$ nonconsecutive integers from the set $[n]$.

\begin{lemma}\label{Count nonconsecutive}
    The number of ways to select $k$ nonconsecutive integers from the set $[n]$ is 
    $\binom{n+1-k}{k}.$
\end{lemma}

Together, Proposition \ref{Fib. Theorem} and Lemma \ref{Count nonconsecutive} imply the combinatorial identity
\[\sum_{k\geq 0}\binom{n+1-k}{k}=F_{n+2}.\]

We are now able to enumerate the elements of $\calA(\hroot,\mu)$ for any $\mu\in\Phi^+$.

\begin{corollary}\label{Cardinality of Fibonacci} 
    Fix $1\leq i\leq j\leq r$ and let $\hroot$ be the highest root of $\mathfrak{sl}_{r+1}(\C)$.
    Then 
    \[|\calA(\hroot, \alpha_i + \cdots + \alpha_j)| = F_i \cdot F_{r-j+1},\]
    where $F_n$ denotes the $n^{\text{th}}$ Fibonacci number.
\end{corollary}

\begin{proof}
Let $1\leq i\leq j\leq r$.
By Theorem \ref{Weyl element thm}, we know $\sigma \in \calA(\hroot, \alpha_i+\alpha_{i+1}+\cdots+\alpha_j)$
if and only if 
$\sigma=1$ or $\sigma=s_{\ell_1} \cdots s_{\ell_k} s_{\ell_{k+1}} \cdots s_{\ell_t}$ for some collection of nonconsecutive integers 
$2\leq \ell_1, \ell_2, \ldots, \ell_k \leq i-1$ and $j+1 \leq \ell_{k+1}, \ldots, \ell_{t} \leq r-1$.

Selecting a subset consisting of nonconsecutive integers from the set $\{2, \ldots, i-1\}$
is the same as selecting a subset consisting of nonconsecutive integers from the set $\{1, \ldots, i-2\}$. 
By Proposition~\ref{Fib. Theorem}, there are $F_i$ such subsets.
The same process can be done for 
selecting a subset consisting of nonconsecutive integers from the set $\{j+1, \ldots, r-1\}$ to get $F_{r-j+1}$ such subsets.

The choices for subsets  
$\{2, \ldots, i-1\}$ and $\{j+1, \ldots, r-1\}$ 
are disjoint and independent of each other. 
Thus, the total number of ways to construct an element of $\calA(\hroot, \alpha_i+\cdots+\alpha_j)$
is given by the product  $F_{i}\cdot F_{r-j+1}$.
Therefore, $|\calA(\hroot, \alpha_i + \cdots + \alpha_j)| = F_i \cdot F_{r-j+1}$.
\end{proof}

\section{The \texorpdfstring{$q$}{q}-Analog of Kostant's  Partition Function}\label{q-analog of KPF Section}
 
Now we are going to give formulas for the value of the $q$-analog of Kostant's partition function when the input is $\sigma(\hroot+\rho)-\rho-\mu$ with $\mu\in \Phi^+$ and $\sigma\in\calA(\hroot,\mu)$. 
To begin, we give some counts for the number of $\sigma\in\calA(\hroot,\mu)$ with a fixed length.

\begin{lemma}\label{Sum for r-j}
    If $\mu = \alpha_1 + \cdots + \alpha_j$ with $1\leq j \leq r-1$, then 
\begin{align}
    |\{\sigma \in \calA(\lambda, \mu) \;| \;\ell(\sigma)=k \mbox{ and $\sigma$ contains $s_{j+1}$}\}|&=\binom{r-j-2-k}{k} \text{and}\\
    |\{\sigma \in \calA(\lambda, \mu) \;|\; \ell(\sigma)=k \mbox{ and $\sigma$ does not contain $s_{j+1}$}\}|&=\binom{r-j-1-k}{k}.
\end{align}
If $\mu = \alpha_i + \cdots + \alpha_r$ with $1< i \leq r$, then 
\begin{align}
    |\{\sigma \in \calA(\lambda, \mu) \;| \;\ell(\sigma)=k \mbox{ and $\sigma$ contains $s_{i-1}$}\}|&=
    \binom{i-3-k}{k} \text{and}\\
    |\{\sigma \in \calA(\lambda, \mu) \;|\; \ell(\sigma)=k \mbox{ and $\sigma$ does not contain $s_{i-1}$}\}|&=
    \binom{i-2-k}{k}.
\end{align}

\end{lemma}

\begin{proof}
    Let $\mu = \alpha_1 + \cdots + \alpha_j$ with $1\leq j < r$. 
    Let $\sigma$ contain $s_{j+1}$.
    Then we select $k$ nonconsecutive integers from the set $\{j+3, \cdots, r-1\}$.
    Re-index the elements to be the set $\{1, \ldots, r-j-3\}$.
    By Lemma \ref{Count nonconsecutive}, we get $\binom{r-j-2-k}{k}$.
    Thus, $|\{\sigma \in \calA(\lambda, \mu) \;| \;\ell(\sigma)=k \mbox{ and $\sigma$ contains $s_{j+1}$}\}|=\binom{r-j-2-k}{k}$.\\
    Let $\sigma$ not contain the element $s_{j+1}$.
    Then we select $k$ nonconsecutive integers from the set\\ $\{j+2,\ldots, r-1\}$.
    Re-index the elements to get the set $\{1, \ldots, r-j-2\}$.
    By Lemma \ref{Count nonconsecutive}, we get $\binom{r-j-1-k}{k}$.
    Thus, $|\{\sigma \in \calA(\lambda, \mu) \;|\; \ell(\sigma)=k \mbox{ and $\sigma$ does not contain $s_{j+1}$}\}|=\binom{r-j-1-k}{k}$.

The proof for the case where $\mu = \alpha_i + \cdots + \alpha_r$ with $1< i \leq r$ is analogous.
\end{proof}

\begin{lemma}\label{bound for r-j}
   If $\mu = \alpha_1 + \cdots + \alpha_j$ with $1\leq j \leq r-1$, then
   \begin{align}
    \max\{\sigma \in \calA(\lambda, \mu) \;| \;\ell(\sigma)=k \mbox{ and $\sigma$ contains $s_{j+1}$}\}&= \left\lfloor \frac{r-j-2}{2} \right\rfloor\mbox{   and} \nonumber\\
    \max\{\sigma \in \calA(\lambda, \mu) \;|\; \ell(\sigma)=k \mbox{ and $\sigma$ does not contain $s_{j+1}$}\}&=\left\lfloor \frac{r-j-1}{2} \right\rfloor \nonumber
    \end{align}
       If $\mu = \alpha_i + \cdots + \alpha_r$ with $2 \leq i \leq r$, then
   \begin{align}
    \max\{\sigma \in \calA(\lambda, \mu) \;| \;\ell(\sigma)=k \mbox{ and $\sigma$ contains $s_{i-1}$}\}&= \left\lfloor \frac{i-3}{2} \right\rfloor\mbox{      and} \nonumber\\
    \max\{\sigma \in \calA(\lambda, \mu) \;|\; \ell(\sigma)=k \mbox{ and $\sigma$ does not contain $s_{i-1}$}\}&=\left\lfloor \frac{i-2}{2} \right\rfloor. \nonumber
\end{align}
\end{lemma}

\begin{proof}
    Let $\mu = \alpha_1 + \cdots + \alpha_j$ with $1\leq j \leq r-1$ and $\ell(\sigma)=k$.
    Suppose $\sigma$ contains $s_{j+1}$.
    Then the smallest 
    integer greater
    than and nonconsecutive to $j+1$ that can be selected is $j+3$.
    Then we select $k$ nonconsecutive integers from the set $\{j+3, \ldots, r-1\}$, which is the same count as if we select $k$ nonconsecutive integers from the set to get $\{1,\ldots, r-j-3\}$.
    The maximum number of nonconsecutive numbers that can be selected from the set $\{1,\ldots, r-j-3\}$ is $\left\lfloor\frac{r-j-2}
    {2}\right\rfloor$.

    Suppose $\sigma$ does not contains $s_{j+1}$.
    Then the smallest
    integer greater
    than and nonconsecutive to $j+1$ that can be selected is $j+2$.
    Then we select $k$ nonconsecutive integers from the set $\{j+2, \ldots, r-1\}$, which is the same count as if we select $k$ nonconsecutive integers from the set $\{1,\ldots, r-j-2\}$.
    The maximum number of nonconsecutive numbers that can be selected from the set $\{1,\ldots, r-j-2\}$ is $\left\lfloor\frac{r-j-1}{2}\right\rfloor$.

    The case where $\mu = \alpha_i + \cdots + \alpha_r$ with $2 \leq i \leq r$ and $\ell(\sigma)=k$ is analogous.\qedhere
\end{proof}

\begin{definition}
For any weight $\xi=c_1\alpha_1+c_2\alpha_2+\cdots+c_r\alpha_r$ with $c_i\in\mathbb{N}$, let $h(\xi)=\sum_{i=1}^rc_i$ denote the \textbf{height} of $\xi$. 
\end{definition}
The following technical lemma helps in giving formulas for the value of the $q$-analog.

\begin{lemma}\label{reindex poly}
    Let $s,t\geq 1$. If $\beta=\alpha_1+\alpha_2+\cdots +\alpha_s$ and $\gamma=\alpha_t+\alpha_{t+1}+\cdots +\alpha_{s+t-1}$, then \[\wp_q(\beta)=\wp_q(\gamma)=q(1+q)^{s-1}.\]
\end{lemma}
\begin{proof}
    Let $\beta=\alpha_1+\alpha_2+\cdots +\alpha_s$ and $\gamma=\alpha_t+\alpha_{t+1}+\cdots +\alpha_{s+t-1}$.
    Notice that $h(\beta)=s$ and $h(\gamma)=(s+t-1)-t+1=s$.
    By \cite[Proposition 3.1]{PHThesisPublication}, we know \[\wp_q(\beta)=\wp_q(\alpha_1+\alpha_2+\cdots +\alpha_s)=q(1+q)^{s-1}.\]

    Now we compute $\wp_q(\gamma)$.
         In general, we can write $\gamma$ as a sum of $y\in[s]$ many positive roots by selecting where to place the parentheses enclosing consecutively indexed positive roots. 
         This can be done in $\binom{s-1}{y-1}$ ways. 
    Thus, we get
    \begin{align}
    \wp_q(\gamma)=q^s+(s-1)q^{s-1}+\binom{s-1}{2}q^{s-2}+\cdots+\binom{s-1}{y-1}q^{y}+\cdots+(s-1)q^2+q\nonumber&=\sum_{y=1}^{s}\binom{s-1}{y-1}q^{y}.\nonumber
    \end{align}
    Reindex $y$ to be $y+1$ and use the Binomial Theorem to get 
    \begin{align}
    \sum_{y=0}^{s-1}\binom{s-1}{y}q^{y+1}\nonumber=q\sum_{y=0}^{s-1}\binom{s-1}{y}q^{y} 
    = q(1+q)^{s-1}.\nonumber
    \end{align}
    Therefore, $\wp_q(\beta)=\wp_q(\gamma)$.
\end{proof}

Notice that if the height of $\beta$ and $\gamma$ are the same, then by Lemma \ref{reindex poly}, the resulting polynomials $\wp_q(\beta)$ and $\wp_q(\gamma)$ are the same. 
Namely, 
reindexing a positive root of the same height does not change the value of the $q$-analog on it.

The following propositions give the value of the $q$-analog of Kostant's partition function when the input is $\sigma(\hroot+\rho)-\rho-\mu$, with $\mu \in \Phi^+$ and $\sigma \in \calA(\tilde{\alpha}, \mu)$.
Similar to Proposition 6.1 in \cite{Kevin}, the following propositions are closely related.
However, they are broken up into different cases based on how $\mu$ is constructed.
Thus, we adapt the techniques from \cite{Kevin} in each of the proofs when needed.

\begin{proposition}\label{prop 1 through r}
    Let $\mu = \alpha_1 + \cdots + \alpha_r$. 
    If $\sigma\in\calA(\hroot,\mu)$, 
    then $\wp_q(\sigma(\lambda+\rho)-\rho-\mu) =1$.
\end{proposition}

\begin{proof}
    By Theorem \ref{Nonconsecutive_Integers}, in order for $\sigma\in\calA(\hroot,\mu)$, it must be that $\sigma=s_{\ell_1}s_{\ell_2}\cdots s_{\ell_a}s_{\ell_{a+1}}s_{\ell_{a+2}}\cdots s_{\ell_{k}}$ where
    $2\leq \ell_1<\ell_2<\cdots<\ell_a\leq i-1$ and $j+1\leq \ell_{a+1}<\ell_{a+2}<\cdots<\ell_k\leq r-1$
    are all nonconsecutive integers. 
    We cannot use $s_{1}, \ldots, s_{r}$ in any subword of $\sigma$.
    This implies that $\sigma=1$ and $\calA(\hroot,\mu)=\{1\}$.
    Hence
    $\wp_q(\sigma(\lambda+\rho)-\rho-\mu)=\wp_q(1(\tilde{\alpha}+\rho)-\rho -(\tilde{\alpha}))= \wp_q(0)=1$,
    as desired.
\end{proof}

\begin{proposition}\label{prop 1 through j}
~
    \begin{enumerate}
        \item[(1)] If $\mu = \alpha_1 + \cdots + \alpha_j$, where $1\leq j \leq  r-1$ and if 
     $\sigma\in\calA(\hroot,\mu)$,  
    then
    \begin{equation}
    \begin{aligned}
        \wp_q(\sigma(\hroot+\rho)-\rho-\mu) &=\begin{cases}
             q^{\ell(\sigma)}(1+q)^{r-h(\mu)-2\ell(\sigma)} &\mbox{if $s_{j+1}$ is contained in $\sigma$}\\
              q^{1+\ell(\sigma)}(1+q)^{r-h(\mu)-2\ell(\sigma)-1} &\mbox{otherwise.}
        \end{cases}\\
    \end{aligned}
    \end{equation}
    
    \item[(2)] If $\mu = \alpha_i + \cdots + \alpha_r$, where $2\leq i\leq r$, and if $\sigma\in\calA(\hroot,\mu)$,  
    then
    \begin{equation}
    \begin{aligned}
        \wp_q(\sigma(\hroot+\rho)-\rho-\mu) &=\begin{cases}
              q^{\ell(\sigma)}(1+q)^{r-h(\mu)-2\ell(\sigma)} 
             &\mbox{if $s_{i-1}$ is containted in $\sigma$}\\
            q^{1+\ell(\sigma)}(1+q)^{r-h(\mu)-2\ell(\sigma)-1}
            &\mbox{otherwise.}
        \end{cases}\\
    \end{aligned}
    \end{equation}
    \end{enumerate}
    
\end{proposition}

\begin{proof}
    Let $\mu = \alpha_1 + \cdots + \alpha_j$.
    By Theorem \ref{Nonconsecutive_Integers}, if $\sigma\in\calA(\hroot,\mu)$,  then $\sigma=s_{\ell_1}s_{\ell_2}\cdots s_{\ell_{k}}$ where \\
    $j+1\leq \ell_{1}<\ell_{2}<\cdots<\ell_k\leq r-1$
    are all nonconsecutive integers.
    
    Suppose $\ell_1=j+1$, i.e. $s_{j+1}$ is contained in $\sigma$.
    Then take $\wp_q(\sigma(\tilde{\alpha}+\rho)-\rho - (\alpha_1+\cdots+\alpha_j))$ and simplify $\sigma(\tilde{\alpha}+\rho)-\rho - (\alpha_1+\cdots+\alpha_j)$ by segmenting the sum of $\alpha$'s by indicating where there was a $s_{x}$ in $\sigma$. 
    This yields
    \begin{align}
    \wp_q([\alpha_{\ell_1+1}+\cdots+\alpha_{\ell_{2}-1}]
        +[\alpha_{\ell_{2}+1}+\cdots+\alpha_{\ell_{3}-1}]
        +\cdots
        +[\alpha_{\ell_{k}+1}+\cdots+\alpha_{r}]).\label{Case 2.a.1}
    \end{align}
    Substituting $\ell_1=j+1$ into equation \eqref{Case 2.a.1} yields
\begin{align}\wp_q([\alpha_{j+2}+\cdots+\alpha_{\ell_{2}-1}]
    +[\alpha_{\ell_{2}+1}+\cdots+\alpha_{\ell_{3}-1}]
    +\cdots
    +[\alpha_{\ell_{k}+1}+\cdots+\alpha_{r}]).\label{Case 2.a.2}
\end{align}

    Since the subsets of the positive roots used to write each of the sums within the brackets are pairwise disjoint, we can rewrite equation \eqref{Case 2.a.2} as the product
    \begin{align}
    \wp_q(\alpha_{j+2}+\cdots+\alpha_{\ell_{2}-1})\cdot\wp_q(\alpha_{\ell_{2}+1}+\cdots+\alpha_{\ell_{3}-1})\cdots
    \wp_q(\alpha_{\ell_{k}+1}+\cdots+\alpha_{r})\label{Case 2.a.3}.
    \end{align}
    By Lemma \ref{reindex poly}, we get that equation \eqref{Case 2.a.3} simplifies to
    \begin{align}
    q(1+q)^{(\ell_2-j-2)-1} q(1+q)^{(\ell_3-\ell_2-1)-1}
    \cdots q(1+q)^{(r-\ell_{k})-1}
    &= q^{k}(1+q)^{r-j-2(k)}
    = q^{\ell(\sigma)}(1+q)^{r-j-2\ell(\sigma)}\label{Case 2.a.6}
\end{align}
where equation \eqref{Case 2.a.6} holds by simplifying the exponents for $q$ and $(1+q)$
and noting that there are $k = \ell(\sigma)$ terms.

Now suppose $\ell_1 \neq j+1$, i.e. $s_{j+1}$ is not contained in $\sigma$.
Then take $\wp_q(\sigma(\tilde{\alpha}+\rho)-\rho - (\alpha_1+\cdots+\alpha_j))$ and simplify $\sigma(\tilde{\alpha}+\rho)-\rho - (\alpha_1+\cdots+\alpha_j)$ by segmenting the sum of $\alpha$'s by indicating where there was a $s_{x}$ in $\sigma$.
This yields
\begin{align}
\wp_q([\alpha_{j+1}+\cdots+\alpha_{\ell_{1}-1}]+[\alpha_{\ell_{1}+1}+\cdots+\alpha_{\ell_{2}-1}]+
     \cdots+[\alpha_{\ell_{k}+1}+\cdots+\alpha_{r}]).\label{Case 2.b.1}
\end{align}

    Since the subsets of the positive roots  used to write each of the sums within the brackets are pairwise disjoint, we can rewrite the equation \eqref{Case 2.b.1} as the product
\begin{align}
    \wp_q(\alpha_{j+1}+\cdots+\alpha_{\ell_1-1})\cdot \wp_q(\alpha_{\ell_1+1}+\cdots+\alpha_{\ell_2-1}) \cdots
    \wp_q(\alpha_{\ell_k+1}+\cdots+\alpha_{r}).\label{Case 2.b.2}
\end{align}
By Lemma \ref{reindex poly}, we get that equation \eqref{Case 2.b.2} simplifies to
    \begin{align}
    q(1+q)^{(\ell_1-j-1)-1} \cdot q(1+q)^{(\ell_2-\ell_1-1)-1}
    \cdots q(1+q)^{(r-\ell_{k})-1}
    &= q^{k+1}(1+q)^{r-j-2(k)-1}\nonumber\\
    &= q^{\ell(\sigma)+1}(1+q)^{r-j-2\ell(\sigma)-1}\label{Case 2.b.5}
\end{align}
where equation \eqref{Case 2.b.5} holds by simplifying the exponents for $q$ and $(1+q)$.

Thus, since $h(\mu)=j$, we arrive at
        \[\wp_q(\sigma(\lambda+\rho)-\rho-\mu) =\begin{cases}
             q^{\ell(\sigma)}(1+q)^{r-h(\mu)-2\ell(\sigma)} &\mbox{if $\ell_1 = j+1$}\\
              q^{\ell(\sigma)+1}(1+q)^{r-h(\mu)-2\ell(\sigma)-1} &\mbox{if $\ell_1 \neq j+1$}.
        \end{cases}\]

        The proof for statement (2) is analogous.
\end{proof}

\begin{proposition}\label{prop i through j}
     Let $\mu = \alpha_i + \cdots + \alpha_j$, where $2 \leq i\leq j \leq r-1$. 
     If $\sigma\in\calA(\hroot,\mu)$,
     then
    \begin{align}
        \wp_q(\sigma(\hroot+\rho)-\rho-\mu) &=\begin{cases}
             q^{2+\ell(\sigma)}(1+q)^{r-h(\mu)-2\ell(\sigma)-2} &\mbox{if $s_{i-1}$ and $s_{j+1}$ are not containted in $\sigma$}\\
             q^{1+\ell(\sigma)}(1+q)^{r-h(\mu)-2\ell(\sigma)-1} &\mbox{if $s_{i-1}$ is contained in $\sigma$ but $s_{j+1}$ is not}\\
             q^{1+\ell(\sigma)}(1+q)^{r-h(\mu)-2\ell(\sigma)-1} &\mbox{if $s_{i-1}$ is not contained in $\sigma$ but $s_{j+1}$ is}\\
             q^{\ell(\sigma)}(1+q)^{r-h(\mu)-2\ell(\sigma)} & \mbox{if $s_{i-1}$ and $s_{j+1}$ are contained in $\sigma$.}
        \end{cases}
    \end{align}
\end{proposition}

\begin{proof}
    Let $\mu = \alpha_i + \cdots + \alpha_j$, where $2\leq i\leq j \leq r-1$ and note $h(\mu)=j-i+1$.

By Theorem \ref{Nonconsecutive_Integers},
since $\sigma\in\calA(\hroot,\mu)$, then  $\sigma=s_{\ell_1}s_{\ell_2}\cdots s_{\ell_a}s_{\ell_{a+1}}s_{\ell_{a+2}}\cdots s_{\ell_{k}}$ where \\
$2\leq \ell_1<\ell_2<\cdots<\ell_a\leq i-1$ and $j+1\leq\ell_{a+1}<\ell_{a+2}<\cdots<\ell_k\leq r-1$
are all nonconsecutive integers.
Let
$\sigma =\sigma_1\sigma_2$ where $\sigma_1=s_{\ell_1}s_{\ell_2}\cdots s_{\ell_a}$ and $\sigma_2=s_{\ell_{a+1}}s_{\ell_{a+2}}\cdots s_{\ell_{k}}$. 
Since the $\ell_x$'s appearing in $\sigma $ are nonconsecutive, then the $s_{\ell_x}$'s can appear in $\sigma$ in any order. This implies  $\sigma_1\sigma_2=\sigma_2\sigma_1$.
Then 
    \begin{align*}
        \wp_q(\sigma(\lambda+\rho)-\rho-\mu)
        &= \wp_q\left(\hroot+\rho - \left(\sum_{ s_{\ell_b}\in \sigma_1}\alpha_{\ell_b} \right)- \left(\sum_{s_{\ell_d}\in \sigma_2}\alpha_{\ell_d}\right) - \rho -\mu\right)\\
        &=\wp_q\left(\sum_{k=1}^r\alpha_k-\left(\sum_{ s_{\ell_b}\in \sigma_1}\alpha_{\ell_b} \right)- \left(\sum_{s_{\ell_d}\in \sigma_2}\alpha_{\ell_d}\right) -\sum_{k=i}^j\alpha_k\right).
        \end{align*}
    Simplifying $\sum_{k=1}^r\alpha_k-\left(\sum_{s_{\ell_b}\in \sigma_1}\alpha_{\ell_b} \right)- \left(\sum_{s_{\ell_d}\in \sigma_2}\alpha_{\ell_d}\right) -\sum_{k=i}^j\alpha_k$ by segmenting the sums and indicating where there was a $s_{x}$ in $\sigma$ yields      
    {\small\begin{align*}
        \wp_q([\alpha_1+\cdots+\alpha_{\ell_1-1}]+
        \cdots +[\alpha_{\ell_a+1}+\cdots+\alpha_{i-1}]
        + [\alpha_{j+1}+\cdots+\alpha_{\ell_{a+1}-1}]
        +\cdots+[\alpha_{\ell_{k}+1}+\cdots+\alpha_r]).
    \end{align*}}
        Since the subsets of the positive roots  used to write each of the sums within the brackets are pairwise disjoint, we can rewrite the equation as the product
  {\small\begin{align}
        \wp_q&([\alpha_1+\cdots+\alpha_{\ell_1-1}]+
        \cdots +[\alpha_{\ell_a+1}+\cdots+\alpha_{i-1}])\cdot\wp_q(
[\alpha_{j+1}+\cdots+\alpha_{\ell_{a+1}-1}]
        +\cdots+[\alpha_{\ell_{k}+1}+\cdots+\alpha_r])\nonumber\\
        &=\wp_q(\sigma_1(\hroot+\rho)-\rho-(\alpha_i+\cdots+\alpha_r))
        \cdot
        \wp_q(\sigma_2(\hroot+\rho)-\rho-(\alpha_1+\cdots+\alpha_{j})).\label{main eq}
\end{align}}

Suppose
$s_{i-1}$ and $s_{j+1}$ are not contained in $\sigma$.
Proposition \ref{prop 1 through j} allows us to simplify equation~\eqref{main eq} to get
\begin{align}
        \wp_q(\sigma(\hroot+\rho)-\rho-\mu)  &=q^{\ell(\sigma_1)+1}(1+q)^{i-2-2\ell(\sigma_1)}q^{\ell(\sigma_2)+1}(1+q)^{r-j-2\ell(\sigma)-1}\nonumber\\
        &=q^{2+\ell(\sigma)}(1+q)^{r-h(\mu)-2\ell(\sigma)-2}\label{Case 1.a.3}.
    \end{align}

Equation \eqref{Case 1.a.3} is achieved by combining the exponents for the $q$'s and the $(1+q)$'s and by recalling the fact that $\ell(\sigma_1)+\ell(\sigma_2)=\ell(\sigma)$ and $h(\mu)=j-i+1$.

Suppose 
$s_{i-1}$ is contained in $\sigma$ but $s_{j+1}$ is not contained in $\sigma$.
Proposition \ref{prop 1 through j} allows us to simplify equation \eqref{main eq} to get
\begin{align}
        \wp_q(\sigma(\hroot+\rho)-\rho-\mu)
        &=q^{\ell(\sigma_1)}(1+q)^{i-1-2\ell(\sigma_1)}q^{\ell(\sigma_2)+1}(1+q)^{r-j-2\ell(\sigma)-1}\nonumber\\
        &=q^{1+\ell(\sigma)}(1+q)^{r-h(\mu)-2\ell(\sigma)-1}.\label{Case 1.b.3}
    \end{align}
    
Equation \eqref{Case 1.b.3} is achieved by combine the exponents for the $q$'s and the $(1+q)$'s and by recalling the fact that $\ell(\sigma_1)+\ell(\sigma_2)=\ell(\sigma)$ and $h(\mu)=j-i+1$.

The case where  
$s_{i-1}$ is not contained in $\sigma$ and $s_{j+1}$ is contained in $\sigma$, is analogous to the case where 
$s_{i-1}$ is contained in $\sigma$ but $s_{j+1}$ is not contained in $\sigma$,
and hence we omit the argument.

Suppose
$s_{i-1}$ and $s_{j+1}$ are contained in $\sigma$.
Proposition \ref{prop 1 through j} allows us to simplify equation \eqref{main eq} to get
\begin{align}
        \wp_q(\sigma(\hroot+\rho)-\rho-\mu)
        &=q^{\ell(\sigma_1)}(1+q)^{(i-1)-2\ell(\sigma_1)}q^{\ell(\sigma_2)}(1+q)^{(r-j)-2\ell(\sigma_2)}\nonumber\\
        &=q^{\ell(\sigma)}(1+q)^{r-h(\mu)-2\ell(\sigma)}.\label{Case 1.d.3}
    \end{align}

Equation \eqref{Case 1.d.3} is achieved by combining the exponents for the $q$'s and the $(1+q)$'s and by recalling the fact that $\ell(\sigma_1)+\ell(\sigma_2)=\ell(\sigma)$ and $h(\mu)=j-i+1$.

Therefore, 
\[\wp_q(\sigma(\hroot+\rho)-\rho-\mu) =\begin{cases}
             q^{2+\ell(\sigma)}(1+q)^{r-h(\mu)-2\ell(\sigma)-2} &\mbox{if $s_{i-1}$ and $s_{j+1}$ are not containted in $\sigma$}\\
             q^{1+\ell(\sigma)}(1+q)^{r-h(\mu)-2\ell(\sigma)-1} &\mbox{if $s_{i-1}$ is contained in $\sigma$ but $s_{j+1}$ is not}\\
             q^{1+\ell(\sigma)}(1+q)^{r-h(\mu)-2\ell(\sigma)-1} &\mbox{if $s_{i-1}$ is not contained in $\sigma$ but $s_{j+1}$ is}\\
             q^{\ell(\sigma)}(1+q)^{r-h(\mu)-2\ell(\sigma)} & \mbox{if $s_{i-1}$ and $s_{j+1}$ are contained in $\sigma$.}
        \end{cases}\]
\end{proof}

\section{The \texorpdfstring{$q$}{q}-Analog of Kostant's Weight Multiplicity Formula of a Positive Root is a Power of \texorpdfstring{$q$}{q}}\label{Main Conjecture for q-analog Section}

We have all the information needed,
we are now ready to prove Theorem \ref{Main Conjecture}, which we restate for ease of reference.

\begin{theorem}\label{Main Conjecture}
    Fix $1\leq i \leq j\leq r$ and let $\hroot$ be the highest root of $\mathfrak{sl}_{r+1}(\C)$. 
    Then \[m_q(\tilde{\alpha}, \mu)= q^{r-h(\mu)},\]
    where 
    $h(\mu)=j-i+1$ denotes the height of $\mu$. 
\end{theorem}

\begin{proof}
    Let
    $\mu=\alpha_i+\alpha_{i+1}+\cdots+\alpha_j\in \Phi^+$
    with $1\leq i\leq j\leq r$,
    which has height $h(\mu)=j-i+1$.
    By definition
    \[m_q(\hroot,\mu)=\sum_{\sigma\in W}(-1)^{\ell(\sigma)}\wp_q(\sigma(\hroot+\rho)-\rho-\mu).\]
     By Theorem \ref{Nonconsecutive_Integers}, we can reduce the computing to be only on the Weyl alternation set. Namely,
    \[m_q(\hroot,\mu)=\sum_{\sigma\in \calA(\hroot,\mu)}(-1)^{\ell(\sigma)}\wp_q(\sigma(\hroot+\rho)-\rho-\mu).\]
For any $x\in[r]$, let
\begin{itemize}
    \item $N_x=\{\sigma\in\calA(\hroot,\mu)\;|\;\mbox{$\sigma$ does not contain $s_x$ in its reduced word}\}$
    \item $Y_x=\{\sigma\in\calA(\hroot,\mu)\;|\;\mbox{$\sigma$ contains $s_x$ in its reduced word}\}$.
\end{itemize}
For each case, we partition the set $\calA(\hroot, \mu)$ into disjoint subsets of $N_x$ and $Y_x$ according to the construction of $\mu$. 

\noindent \textbf{Case 1}:
Fix $\mu= \alpha_i+\cdots+\alpha_j$ where $2\leq i\leq j\leq r-1$.
To begin, we follow the same construction of equation \eqref{main eq}, which we restate here:
\begin{align}
    \wp_q(\sigma(\hroot+\rho)-\rho-\mu)
    =\wp_q(\sigma_1(\hroot+\rho)-\rho-(\alpha_i+\cdots+\alpha_r))
        \cdot
        \wp_q(\sigma_2(\hroot+\rho)-\rho-(\alpha_1+\cdots+\alpha_{j})) \label{Case 1 wp}
\end{align}
where $\ell(\sigma)=k$, $\sigma=\sigma_1\sigma_2$, $\sigma_1 = s_{\ell_1}s_{\ell_2}\cdots s_{\ell_a}$ where $2 \leq \ell_1 < \ell_2 < \cdots < \ell_a \leq i-1$, and
$\sigma_2 = s_{\ell'_1}s_{\ell'_2}\cdots s_{\ell'_b}$ where $j+1 \leq \ell'_1 < \ell'_2 < \cdots < \ell'_b \leq r-1$ are all nonconsecutive integers. Note $\ell(\sigma_1)=a$, $\ell(\sigma_2)=b$, and $\ell(\sigma)=k=a+b$.
By construction, we have 
\begin{align}
    \calA(\hroot,\mu)=(N_{i-1}\cap N_{j+1})\sqcup (N_{i-1} \cap Y_{j+1})\sqcup(Y_{i-1} \cap N_{j+1})\sqcup(Y_{i-1}\cap Y_{j+1})\label{split into 4 sums}
\end{align}
where $\sigma$ is considered to be in the set {$N_{i-1}\cap N_{j+1}$, $N_{i-1} \cap Y_{j+1}$, $Y_{i-1} \cap N_{j+1}$, or $Y_{i-1}\cap Y_{j+1}$.}

\begin{itemize}
    \item Let $\sigma\in N_{i-1}\cap N_{j+1}$ such that $\ell(\sigma)=\ell(\sigma_1)+\ell(\sigma_2)=k$ where we let $\ell(\sigma_1)=a$ and $\ell(\sigma_2)=b$. 
    Then 
\begin{align*}
\wp_q(\sigma_1(\hroot+\rho)-\rho-(\alpha_i+\cdots+\alpha_r))
        &=\wp_q\left(\hroot +\rho - \sum_{s_{x} \in \sigma_1}\alpha_{x} - \rho - (\alpha_i + \cdots +\alpha_r)\right)\\ 
        &= \wp_q\left(\alpha_1+\cdots+\alpha_{i-1}+\gamma-\sum_{s_{x} \in \sigma_1}\alpha_{x}-\gamma\right)\\
        &=\wp_q(\sigma_1(\beta+\gamma)-\gamma)
\end{align*}
where $\beta$ is the highest root of $\mathfrak{sl}_{i}(\mathbb{C})$, $\gamma$ is the sum of the fundamental weights of $\mathfrak{sl}_{i}(\mathbb{C})$, and $\sigma_1\in\calA(\beta,0)$ \cite[Theorem 2.1]{PHThesisPublication}.

Similarly,
\[
\wp_q(\sigma_2(\hroot+\rho)-\rho-(\alpha_{1}+\cdots+\alpha_j))
        =\wp_q(\sigma_2(\beta'+\gamma')-\gamma')
\]
where $\beta'$ is the highest root of $\mathfrak{sl}_{r-j+1}(\mathbb{C})$, $\gamma'$ is the sum of the fundamental weights of $\mathfrak{sl}_{r-j+1}(\mathbb{C})$, and $\sigma_2\in\calA(\beta',0)$ \cite[Theorem 2.1]{PHThesisPublication}.

By \cite[Proposition 3.2]{PHThesisPublication},
note that 
\begin{align*}
\wp_q(\sigma_1(\beta+\gamma)-\gamma)&=q ^{1 + \ell ( \sigma_1 )} ( 1 + q )^{ (i-1) - 1 -2 \ell( \sigma_1)}\mbox{ and}\\
\wp_q(\sigma_2(\beta'+\gamma')-\gamma')&=q^{1 + \ell ( \sigma_2 )} ( 1 + q )^{ (r-j) -1-2 \ell( \sigma_2)}.
\end{align*}
Given the equality in equation \eqref{Case 1 wp}, note that 
\begin{align}
&\sum_{\sigma\in N_{i-1}\cap N_{j+1}}(-1)^{\ell(\sigma)}\wp_q(\sigma(\hroot+\rho)-\rho-\mu)\nonumber\\
&~=\left(\sum_{\sigma_1\in \calA(\beta,0)}(-1)^{\ell(\sigma_1)}\wp_q(\sigma_1(\beta+\gamma)-\gamma)\right)\cdot \left(\sum_{\sigma_2\in \calA(\beta',0)}(-1)^{\ell(\sigma_2)}\wp_q(\sigma_2(\beta'+\gamma')-\gamma')\right)\nonumber\\
&=m_q(\beta,0)\cdot m_q(\beta',0)\nonumber\\
&=\left(\sum_{t=1}^{i-1}q^t\right)\left(\sum_{t=1}^{r-j}q^t\right),\label{kevin eq 1}
\end{align}
where the equality in \eqref{kevin eq 1} follows from \cite[Theorem 6.1]{Kevin}.

    \item If $\sigma\in N_{i-1}\cap Y_{j+1}$ such that $\ell(\sigma)=\ell(\sigma_1)+\ell(\sigma_2)=k+1$ where we let $\ell(\sigma_1)=a$ and $\ell(\sigma_2)=b+1$.
    Then $\wp_q(\sigma_1(\hroot+\rho)-\rho-(\alpha_i+\cdots+\alpha_r))=\wp_q(\sigma_1(\beta+\gamma)-\gamma)$ like above.
However,
\[
\wp_q(\sigma_2(\hroot+\rho)-\rho-(\alpha_{1}+\cdots+\alpha_j))
        =\wp_q(\sigma_2(\beta'+\gamma')-\gamma')
\]
where $\beta'$ is the highest root of $\mathfrak{sl}_{r-j}(\mathbb{C})$, $\gamma'$ is the sum of the fundamental weights of {$\mathfrak{sl}_{r-j}(\mathbb{C})$}, and $\sigma_2\in\calA(\beta',0)$ \cite[Theorem 2.1]{PHThesisPublication}.

By \cite[Proposition 3.2]{PHThesisPublication},
note that 
\begin{align*}
\wp_q(\sigma_1(\beta+\gamma)-\gamma)&=q ^{1 + \ell ( \sigma_1 )} ( 1 + q )^{ (i-1) - 1 -2 \ell( \sigma_1)} \mbox{ and }\\
\wp_q(\sigma_2(\beta'+\gamma')-\gamma')&=q^{1+\ell(\sigma_2)}(1+q)^{(r-j-1)-1-2\ell(\sigma_2)}.
\end{align*}
Given the equality in equation \eqref{Case 1 wp}, note that 
\begin{align}
&\sum_{\sigma\in N_{i-1}\cap Y_{j+1}}(-1)^{\ell(\sigma)}\wp_q(\sigma(\hroot+\rho)-\rho-\mu)\nonumber\\
&~=\left(\sum_{\sigma_1\in \calA(\beta,0)}(-1)^{\ell(\sigma_1)}\wp_q(\sigma_1(\beta+\gamma)-\gamma)\right)\cdot \left(\sum_{\sigma_2\in \calA(\beta',0)}(-1)^{\ell(\sigma_2)}\wp_q(\sigma_2(\beta'+\gamma')-\gamma')\right)\nonumber\\
&=m_q(\beta,0)\cdot (-1) m_q(\beta',0)\label{case a: -1}
\end{align}
where the equality in \eqref{case a: -1} is satisfied because $\ell(\sigma_2)=b+1$.
Then we get
\begin{align}
(-1)\left(\sum_{t=1}^{i-1}q^t\right)\left(\sum_{t=1}^{r-j-1}q^t\right),\label{kevin eq 2}
\end{align}
where expression \eqref{kevin eq 2} follows from \cite[Theorem 6.1]{Kevin}.

    \item If $\sigma\in Y_{i-1}\cap N_{j+1}$ such that $\ell(\sigma)=\ell(\sigma_1)+\ell(\sigma_2)=k+1$ {where we let $\ell(\sigma_1)=a+1$ and $\ell(\sigma_2)=b$}. 
    Now 
\[\wp_q(\sigma_1(\hroot+\rho)-\rho-(\alpha_i+\cdots+\alpha_r))
        =\wp_q(\sigma_1(\beta+\gamma)-\gamma)
\]
where $\beta$ is the highest root of {$\mathfrak{sl}_{i-1}(\mathbb{C})$}, $\gamma$ is the sum of the fundamental weights of {$\mathfrak{sl}_{i-1}(\mathbb{C})$}, and $\sigma_1\in\calA(\beta,0)$ \cite[Theorem 2.1]{PHThesisPublication}.

Like above, $\wp_q(\sigma_2(\hroot+\rho)-\rho-(\alpha_{1}+\cdots+\alpha_j))=\wp_q(\sigma_2(\beta'+\gamma')-\gamma')$.

By \cite[Proposition 3.2]{PHThesisPublication}
note that 
\begin{align*}
\wp_q(\sigma_1(\beta+\gamma)-\gamma)&=q ^{1+\ell(\sigma_1)}(1+q)^{(i-2)-1-2\ell(\sigma_1)}\mbox{ and}\\
\wp_q(\sigma_2(\beta'+\gamma')-\gamma')&=q^{1+\ell(\sigma_2)}(1+q)^{(r-j)-1-2\ell(\sigma_2)}.
\end{align*}
Given the equality in equation \eqref{Case 1 wp}, note that 
\begin{align}
&\sum_{\sigma\in Y_{i-1}\cap N_{j+1}}(-1)^{\ell(\sigma)}\wp_q(\sigma(\hroot+\rho)-\rho-\mu)\nonumber\\
&=\left(\sum_{\sigma_1\in \calA(\beta,0)}(-1)^{\ell(\sigma_1)}\wp_q(\sigma_1(\beta+\gamma)-\gamma)\right)\cdot \left(\sum_{\sigma_2\in \calA(\beta',0)}(-1)^{\ell(\sigma_2)}\wp_q(\sigma_2(\beta'+\gamma')-\gamma')\right)\nonumber\\
&=(-1) m_q(\beta,0)\cdot m_q(\beta',0)\label{case b: -1}
\end{align}
where the equality in \eqref{case b: -1} is satisfied because $\ell(\sigma_1)=a+1$.
Then we get
\begin{align}
(-1)\left(\sum_{t=1}^{i-2}q^t\right)\left(\sum_{t=1}^{r-j}q^t\right),\label{kevin eq 3}
\end{align}
where expression \eqref{kevin eq 3} follows from \cite[Theorem 6.1]{Kevin}.

    \item If $\sigma\in Y_{i-1}\cap Y_{j+1}$ such that $\ell(\sigma)=\ell(\sigma_1)+\ell(\sigma_2)=k+2$ {where we let $\ell(\sigma_1)=a+1$ and $\ell(\sigma_2)=b+1$}. 
    Like in the above respective corresponding cases,\\ $\wp_q(\sigma_1(\hroot+\rho)-\rho-(\alpha_i+\cdots+\alpha_r))=\wp_q(\sigma_1(\beta+\gamma)-\gamma)$ and \\ $\wp_q(\sigma_2(\hroot+\rho)-\rho-(\alpha_{1}+\cdots+\alpha_j))=\wp_q(\sigma_2(\beta'+\gamma')-\gamma')$.

By \cite[Proposition 3.2]{PHThesisPublication}
note that
\begin{align*}
\wp_q(\sigma_1(\beta+\gamma)-\gamma)&=q ^{1+\ell(\sigma_1)}(1+q)^{(i-2)-1-2\ell(\sigma_1)},\mbox{ and}\\
\wp_q(\sigma_2(\beta'+\gamma')-\gamma')&=q^{1+\ell(\sigma_2)}(1+q)^{(r-j-1)-1-2\ell(\sigma_2)}.
\end{align*}
Given the equality in equation \eqref{Case 1 wp}, note that 
\begin{align}
&\sum_{\sigma\in Y_{i-1}\cap Y_{j+1}}(-1)^{\ell(\sigma)}\wp_q(\sigma(\hroot+\rho)-\rho-\mu)\nonumber\\
&=\left(\sum_{\sigma_1\in \calA(\beta,0)}(-1)^{\ell(\sigma_1)}\wp_q(\sigma_1(\beta+\gamma)-\gamma)\right)\cdot \left(\sum_{\sigma_2\in \calA(\beta',0)}(-1)^{\ell(\sigma_2)}\wp_q(\sigma_2(\beta'+\gamma')-\gamma')\right)\nonumber\\
&=(-1) m_q(\beta,0)\cdot (-1) m_q(\beta',0)\label{case c: -1}
\end{align}
where the equality in \eqref{case c: -1} is satisfied because $\ell(\sigma_1)=a+1$ and $\ell(\sigma_1)=b+1$.
Then we get
\begin{align}
\left(\sum_{t=1}^{i-2}q^t\right)\left(\sum_{t=1}^{r-j-1}q^t\right),\label{kevin eq 4}
\end{align}
where expression \eqref{kevin eq 4} follows from \cite[Theorem 6.1]{Kevin}.
\end{itemize}
    
    Now that we have all cases, we can construct the $m_q(\hroot, \mu)$ from \eqref{split into 4 sums} by adding expression \eqref{kevin eq 1}, expression \eqref{kevin eq 2}, expression \eqref{kevin eq 3}, and expression \eqref{kevin eq 4} to get the following:
    \begin{align*}
        m_q(\hroot, \mu)
        &= \left(\sum_{t=1}^{i-1}q^t\right)\left(\sum_{t=1}^{r-j}q^t\right)-\left(\sum_{t=1}^{i-1}q^t\right)\left(\sum_{t=1}^{r-j-1}q^t\right)-\left(\sum_{t=1}^{i-2}q^t\right)\left(\sum_{t=1}^{r-j}q^t\right)+\left(\sum_{t=1}^{i-2}q^t\right)\left(\sum_{t=1}^{r-j-1}q^t\right)\nonumber\\
        &= q^{r-(j-i+1)}.
    \end{align*}

Thus, $m_q(\hroot, \mu)=q^{r-h(\mu)}$ when $\mu = \alpha_i+\cdots + \alpha_j \in \Phi^+$ with $2\leq i\leq j \leq r-1$ and $h(\mu)=j-i+1$.

\noindent \textbf{Case 2}:
    Suppose $\mu=\alpha_1+\cdots+\alpha_j\in \Phi^+$ where $1\leq j\leq r-1$.
    Notice that $h(\mu)=j$.

     Observe that $\calA(\tilde{\alpha}, \mu)=N_{j+1}\sqcup Y_{j+1}$ is a disjoint union.
Thus,
\begin{align}
    m_q(\tilde{\alpha}, \mu)  =& \sum_{\sigma \in {N_{j+1}}} (-1)^{\ell(\sigma)}\wp_q(\sigma(\tilde{\alpha}+\rho)-\rho - \mu)+ \sum_{\sigma \in {Y_{j+1}}} (-1)^{\ell(\sigma)}\wp_q(\sigma(\tilde{\alpha}+\rho)-\rho - \mu). \label{q-analog for 1 through j}
\end{align}

First, consider $\sum_{\sigma \in {N_{j+1}}} (-1)^{\ell(\sigma)}\wp_q(\sigma(\tilde{\alpha}+\rho)-\rho - \mu)$.
By Lemma \ref{bound for r-j}, we know the maximum number of $\sigma \in \calA(\hroot, \mu)$.
By Lemma \ref{Sum for r-j}, we know there are $\binom{r-j-1-k}{k}$ elements $\sigma \in N_{j+1}$ of length $k$ and that the maximum length of any $\sigma\in N_{j+1}$ is $\left\lfloor\frac{r-j-1}{2}\right\rfloor$.
Proposition \ref{prop 1 through j} gives the value of the $q$-analog $\wp_q(\sigma(\tilde{\alpha}+\rho)-\rho - \mu)$ for $\sigma\in N_{j+1}$.
Thus, if $\sigma \in N_{j+1}$ and we use \cite[Proposition 3.2]{PHThesisPublication}, then the first sum in equation \eqref{q-analog for 1 through j} becomes 
\begin{align}
    \sum_{k=0}^{\left\lfloor \frac{r-j-1}{2}\right\rfloor}(-1)^{k}\binom{r-j-1-k}{k}q^{k+1}(1+q)^{r-j-1-2k}=\sum_{t=1}^{r-j}q^t.\label{expression for N_{j+1}}
\end{align}

Similarly, consider $\sum_{\sigma \in {Y_{j+1}}} (-1)^{\ell(\sigma)}\wp_q(\sigma(\tilde{\alpha}+\rho)-\rho - \mu)$.
By Lemma \ref{bound for r-j},
we know there are $\binom{r-j-2-k}{k}$ elements $\sigma \in Y_{j+1}$ of length $k$ and that the maximum length of any $\sigma\in Y_{j+1}$ is $\left\lfloor\frac{r-j-2}{2}\right\rfloor$.
Proposition \ref{prop 1 through j} gives the value of the $q$-analog $\wp_q(\sigma(\tilde{\alpha}+\rho)-\rho - \mu)$ for $\sigma\in Y_{j+1}$.
Thus, if $\sigma \in Y_{j+1}$ and we use \cite[Proposition 3.2]{PHThesisPublication}, then the second sum in equation \eqref{q-analog for 1 through j} becomes 
\begin{align}
    \sum_{k=0}^{\left\lfloor \frac{r-j-2}{2}\right\rfloor}(-1)^{k+1}\binom{r-j-2-k}{k}q^{k+1}(1+q)^{r-j-2(k+1)}= - \sum_{t=1}^{r-j-1}q^t.\label{Expression for Y_{j+1}}
\end{align}

Calculate $m_q(\hroot, \mu)$ by adding expression \eqref{expression for N_{j+1}} and expression \eqref{Expression for Y_{j+1}} to get
\begin{align}
     \sum_{t=1}^{r-j}q^t- \sum_{t=1}^{r-j-1}q^t = q^{r-j}.
\end{align}
Thus, $m_q(\hroot, \mu)=q^{r-h(\mu)}$ since $h(\mu)=j$.

\noindent \textbf{Case 3}:
Suppose $\mu=\alpha_i+\cdots+\alpha_r\in \Phi^+$ where $2\leq  i\leq r$.
    Notice that $h(\mu)=r-i+1$.

Observe that $\calA(\tilde{\alpha}, \mu)=N_{i-1} \sqcup Y_{i-1}$ is a disjoint union. 
Thus,
\begin{align}
   m_q(\tilde{\alpha}, \mu)= \sum_{\sigma \in N_{i-1}} (-1)^{\ell(\sigma)}\wp_q(\sigma(\tilde{\alpha}+\rho)-\rho - \mu) + \sum_{\sigma \in Y_{i-1}} (-1)^{\ell(\sigma)}\wp_q(\sigma(\tilde{\alpha}+\rho)-\rho - \mu). \label{q-analog for i through r}
\end{align}

First, consider $\sum_{\sigma \in N_{i-1}} (-1)^{\ell(\sigma)}\wp_q(\sigma(\tilde{\alpha}+\rho)-\rho - \mu)$.
By Lemma \ref{bound for r-j}, we know the maximum length of an element $\sigma\in N_{i-1}$ is $\left\lfloor\frac{i-2}{2}\right\rfloor$.
By Lemma \ref{Sum for r-j}, we know there are $\binom{i-2-k}{k}$ elements $\sigma \in N_{i-1}$. 
Proposition \ref{prop 1 through j} gives the value of the $q$-analog  $\wp_q(\sigma(\tilde{\alpha}+\rho)-\rho - \mu)$ for $\sigma\in N_{i-1}$.
Thus if $\sigma \in N_{i-1}$ and we use \cite[Proposition 3.2]{PHThesisPublication}, the first sum in equation \eqref{q-analog for i through r} becomes 
\begin{align}
    \sum_{k=0}^{\left\lfloor \frac{i-2}{2}\right\rfloor}(-1)^{k}\binom{i-2-k}{k}q^{k+1}(1+q)^{i-2-2k}= \sum_{t=1}^{i-1}q^t.\label{expression for N_{i-1}}
\end{align}

Similarly, consider $\sum_{\sigma \in Y_{i-1}} (-1)^{\ell(\sigma)}\wp_q(\sigma(\tilde{\alpha}+\rho)-\rho - \mu)$.
By Lemma \ref{bound for r-j}, we know the maximum length of an element $\sigma\in Y_{i-1}$ is $\left\lfloor\frac{i-3}{2}\right\rfloor$.
By Lemma \ref{Sum for r-j}, we know there are $\binom{i-3-k}{k}$ elements $\sigma \in Y_{i-1}$. 
Proposition \ref{prop 1 through j} gives the value of the $q$-analog  $\wp_q(\sigma(\tilde{\alpha}+\rho)-\rho - \mu)$ for $\sigma\in Y_{i-1}$.
Thus if $\sigma \in Y_{i-1}$ and we use \cite[Proposition 3.2]{PHThesisPublication}, the second sum in equation \eqref{q-analog for i through r} becomes 
\begin{align}
    \sum_{k=0}^{\left\lfloor \frac{i-3}{2}\right\rfloor}(-1)^{k+1}\binom{i-3-k}{k}q^{k+1}(1+q)^{(i-1)-2{(k+1)}}= - \sum_{t=1}^{i-2}q^t.\label{Expression for Y_{i-1}}
\end{align}

Calculate $m_q(\hroot, \mu)$ by adding expression \eqref{expression for N_{i-1}} and expression \eqref{Expression for Y_{i-1}} to get
\begin{align}
    \sum_{t=1}^{i-1}q^t - \sum_{t=1}^{i-2}q^t = q^{i-1}.
\end{align}
Thus, $m_q(\hroot, \mu)=q^{r-h(\mu)}$ since $h(\mu)=r-i+1$.

\noindent \textbf{Case 4}: 
Suppose $\mu=\alpha_1+\cdots+\alpha_r \in \Phi^+$.
Notice that $h(\mu)=r-1+1=r$.
Observe that the elements of $\calA(\tilde{\alpha}, \mu)=\{1\}$.
Thus,    \[
        m_q(\tilde{\alpha}, \mu) = \wp_q(\sigma(\tilde{\alpha}+\rho)-\rho - \mu)= \wp_q(1(\tilde{\alpha}+\rho)-\rho - \tilde{\alpha})= \wp_q(0)
        = 1= q^{r-(r-1+1)}=q^{r-h(\mu)}.
    \]

Therefore, when $r \geq 1$ and $\hroot$ is the highest root of $\mathfrak{sl}_{r+1}(\mathbb{C})$ and $\mu=\alpha_i+\alpha_{i+1}+\cdots+\alpha_j\in \Phi^+$ with $h(\mu)=j-i+1$, then
    $m_q(\tilde{\alpha}, \mu)= q^{r-h(\mu)}$.
\end{proof}

\bibliographystyle{abbrv}
\bibliography{bibliography}

\end{document}